\documentclass[12pt, reqno]{amsart}
\usepackage{amssymb}

\usepackage{hyperref}
\hypersetup{colorlinks=true,linkcolor=black,anchorcolor=black,citecolor=black}
\usepackage{tikz-cd}

\usepackage{xcolor}
\usepackage{longtable}

\usepackage{tikz}
\usepackage[all]{xy}

\usepackage{booktabs}
\usepackage{multirow}

\newcommand{\qs}[2]{\tfrac{1}{#1}(#2)}

\DeclareFontFamily{U}{mathb}{\hyphenchar\font45}
\DeclareFontShape{U}{mathb}{m}{n}{
      <5> <6> <7> <8> <9> <10> gen * mathb
      <10.95> mathb10 <12> <14.4> <17.28> <20.74> <24.88> mathb12
      }{}
\DeclareSymbolFont{mathb}{U}{mathb}{m}{n}
\DeclareMathSymbol{\righttoleftarrow}{3}{mathb}{"FD}

\theoremstyle{plain}
\newtheorem{prop}{Proposition}[]
\newtheorem{theo}[prop]{Theorem}
\newtheorem{conj}[prop]{Conjecture}

\theoremstyle{remark}

\theoremstyle{definition}
\newtheorem{rema}[prop]{Remark}

\newtheorem{exam}[prop]{Example}
\numberwithin{equation}{section}

\def\bP{{\mathbb P}}
\def\bQ{{\mathbb Q}}

\title[Potential density of rational points]{Potential density of rational points\\ on sextic double solids}

\author[Ivan Cheltsov]{Ivan Cheltsov}
\address{\emph{Ivan Cheltsov}
\newline 
University of Edinburgh, Scotland}
\email{I.Cheltsov@ed.ac.uk}

\author[Yuri Tschinkel]{Yuri Tschinkel}
\address{\emph{Yuri Tschinkel}
\newline
  Courant Institute \& Simons Foundation, New York, USA}
\email{tschinkel@cims.nyu.edu}

\date{\today}

\makeatletter
\@namedef{subjclassname@2020}{\textup{2020} Mathematics Subject Classification}
\makeatother


\begin{document}

\begin{abstract}
We provide examples of smooth sextic double solids with potentially dense sets of rational points, over any field of characteristic zero. 
\end{abstract}
	
	\maketitle

\section{Introduction}

Let $k$ be a field of characteristic zero and $X$ a variety over $k$. Following \cite{BT-elliptic}, we say that $X$ satisfies {\em potential density} if there exists a finite extension $k'/k$ such that $k'$-rational points $X(k')$ are Zariski dense. Of particular interest are number fields, or function fields of varieties over algebraically closed ground fields.  

One can ask for geometric conditions ensuring potential density, see e.g.,  
\cite[Conjecture 1.3]{Harris-T}, or the more general framework of \cite{Campana}. 
A special case of these considerations is:

\begin{conj}
\label{conj:main}
Let $X$ be a smooth projective variety over $k$ with ample anticanonical class. 
Then rational points on $X$ are potentially dense. 
\end{conj}

The conjecture is obviously true in dimensions $\le 2$ -- all such varieties are geometrically rational. It is also true for geometrically unirational varieties. 
By the celebrated theorem of Graber--Harris--Starr \cite{GHS}, and its extension to positive characteristic by de Jong--Starr \cite{DS}, the conjecture holds over function fields of curves, over algebraically closed fields. However, it remained notoriously difficult over number fields, in situations when unirationality is unknown, in particular, for smooth Fano threefolds of the following families:
\begin{itemize}
    \item quartic 3-folds,
    \item sextic double solids, and 
    \item Veronese double cones. 
\end{itemize}
Geometric unirationality of {\em some} smooth quartic 3-folds was known classically, going back to Fano. Examples of unirational smooth Veronese double cones were recently produced by \cite{masa}. No unirational examples of smooth sextic double solids are known, at present. 

On the other hand, potential density was settled in the positive for quartics and Veronese double cones in \cite{Harris-T} and \cite{BT-elliptic}, respectively. {\em Singular} sextic double solids were treated in \cite{CP}. The key input in those approaches was existence of {\em birational} maps to elliptic fibrations over rational bases, with 
{\em nontorsion} (or {\em saliently ramified} in the terminology of \cite{BT-elliptic}) multisections with potentially dense sets of rational points; this allowed one to propagate rational points to a Zariski dense subset of the threefold. This technique has been widely applied since its introduction, for example, to show potential density of rational points on elliptic K3 surfaces \cite{BT-K3}, or more recently, 
on Calabi--Yau threefolds with fibrations in abelian surfaces \cite{B-Sho}, or elliptic curves  
\cite{Suzuki-2022}.

However, smooth sextic double solids are {\em never} birational to elliptic or abelian fibrations \cite{CP2010,CP-2007}, in characteristic $\neq 5$. This was perceived as the main obstacle to potential density of rational points on these varieties. Alternative (unsuccessful) approaches to potential density were proposed in \cite{HT-symm}, where it was proved that the symmetric square of a sextic double solid has potentially dense rational points.  

In characteristic 5, there are examples of smooth sextic double solids that are birational to an elliptic fibration, discovered in \cite{CP-2007}, see Example~\ref{exam:char5}. 
Surprisingly, a version of this construction can be made to work in characteristic zero to produce {\em dominant} rational maps from an elliptically fibered threefold to a smooth sextic double solid, see Section~\ref{sect:construct} for more details. This allows us to prove the following result.

\begin{theo}
\label{thm:main}
There exists a codimension-3 (in moduli) family of smooth sextic double solids with potentially dense sets of rational points, over any field of characteristic zero.  
\end{theo}

The proof of this theorem, given in Section~\ref{sect:construct}, is algorithmic and can be used to produce explicit examples. 

\begin{exam}
Let $X\subset\bP(1,1,1,1,3)$ be the sextic double solid given by
$$
x_5^2=x_1^6+x_1x_2^5+x_3^6-x_4^6.
$$
Then $X(\bQ)$ is Zariski dense. 
\end{exam}

In \cite{Fletcher,Reid}, Fletcher and Reid discovered $95$ deformation families of well-formed quasismooth Fano threefold hypersurfaces with terminal singularities in weighted projective spaces. Sextic double solids form one such family (Family No. 3 in the notation of \cite{Fletcher}).  In Section~\ref{sect:95}, we discuss potential density of rational points on Fano 3-folds in these $95$ families.

\

\noindent
{\bf AI disclosure:} We used ChatGPT-6 Astra for producing examples, and proofreading. The paper was written by the authors, who are responsible for its correctness.

\

\noindent
{\bf Acknowledgments:} 
We are grateful to Zhijia Zhang for his interest and collaboration on related problems.
Ivan Cheltsov has been supported by the Simons Collaboration grant \emph{Moduli of varieties}.

\section{The construction}
\label{sect:construct}

We proceed to explain the geometric construction. 
Let 
$$
X:=\{ x_5^2=f_6(x_1,x_2,x_3,x_4) \} \subset \bP(1,1,1,1,3)
$$
be a sextic double solid, i.e., 
$$
\pi: X\to \bP^3\supset S
$$ 
is double cover ramified in a sextic surface $S$, which we assume to be smooth. The idea of the proof   is to produce a 
diagram

 \ 
 
\centerline{
\xymatrix{
\mathcal E \ar@{-->}[r]^{\phi} \ar[d]_{\psi} & X \\
B & 
}
}

\noindent
where 
\begin{itemize}
    \item $\phi$ is a dominant rational map, and 
    \item $\psi$ is an elliptic fibration over a surface, with a {\em saliently ramified} multisection, which has a potentially dense set of  rational points. 
\end{itemize}
Our main observation is that this can be constructed from a curve $C\subset S$ 
whose normalization has genus $\le 1$, e.g., a line in $S$. In hindsight, an avatar of our construction is already present in the following:

\begin{exam}[{\cite[Example~1.5]{CP-2007}}]
\label{exam:char5}
Let $k=\mathbb{F}_5$.
Let $\pi\colon X\to\mathbb{P}^3$ be the double cover ramified along the sextic surface $S$ given by 
\begin{multline*}
x_1^5x_2+x_1^4x_2^2+x_1^2x_2^3x_3-x_2^5x_3-2x_1^4x_3^2+x_1x_3^5+x_2x_3^5\\
+x_1^3x_2^2x_4+2x_1^2x_2^3x_4-x_1x_2x_3^3x_4-x_1x_2x_3^2x_4^2-x_1^2x_2x_4^3\\
+x_1x_2^2x_4^3+x_1^2x_3x_4^3+x_1x_2x_4^4+x_1x_4^5+2x_2x_4^5=0.
\end{multline*}
Then $X$ is nonsingular, and $S$ contains the line $C=\{x_1=x_2=0\}$. 
For sufficiently general $p\in X$, there is a unique hyperplane
$H_p\subset\mathbb{P}^3$ containing $\pi(p)$ and $C$.
The residual quintic curve $Q_p\subset H_p$ given by
$$
S\cap H_p=C\cup Q_p
$$ 
intersects $C$ at a single point $q$ with $\operatorname{ord}_{q}(Q_p|_C)=5$.
The points $\pi(p)$ and $q$ determine a line $L_p\subset \mathbb{P}^3$.
Define a rational map
$$
\Xi_C\colon X\dashrightarrow\operatorname{Gr}(2,4)
$$
by setting $\Xi_C(p)=L_p$.
The image of $\Xi_C$ is isomorphic to $\mathbb{P}^2$,
hence we may assume that $\Xi_C$ is a rational map of $X$ onto $\mathbb{P}^2$.
The resolution of indeterminacy of $\Xi_C$ birationally transforms $X$ into an elliptic fibration.
\end{exam}

Using the approach in this example together with techniques developed in \cite{BT-elliptic}, we obtain:

\begin{theo}
\label{thm:mainn}
Assume that $S$ contains a geometrically integral curve $C$ of degree $d$
such that:
\begin{enumerate} 
    \item the normalization $\tilde C$ of $C$ has genus at most one,
    \item there is a smooth point $p\in C$ such that the curve $T_p(S)\cap S$ has an ordinary double singularity (node) at $p$, where $T_p(S)\subset\mathbb{P}^3$ is the embedded tangent space; 
    \item there exists a line $\ell\subset\mathbb{P}^3$ such that $\ell$ 
    meets $S$ at $p$ with multiplicity exactly 3
    (triple tangency condition), and $\ell\cap S$ consists of the point $p$ (counted with multiplicity $3$) and three other, pairwise distinct, intersection points. 
\end{enumerate} 
Then $X$ has potentially dense rational points.
\end{theo}

\begin{proof}
First, we explain how a curve $C\subset S$ gives rise to an elliptic threefold dominating $X$. To begin with, assume that $C$ is a line, since the proof in this case is inspired by Example~\ref{exam:char5}. 

If $C$ is a line, the planes $\Pi_\tau\subset\mathbb{P}^3$ containing $C$ form a pencil with parameter $\tau\in k\cup\infty$. Taking preimages of these planes on $X$, we obtain the pencil of K3 surfaces $Y_\tau\subset X$. For general $\tau$, we have 
$$
S\cap\Pi_\tau=C+Q_\tau,
$$ 
where $Q_\tau$ is a smooth plane quintic meeting $C$ transversely at five points. 
Thus, $Y_\tau$ has five ordinary double points for general $\tau$ (by Bertini theorem), 
and its minimal resolution is a smooth K3 surface.
Resolving the base locus of the pencil gives a K3 fibration, and we denote by $Y_{k(\tau)}$ its generic fiber. 
Then $Y_{k(\tau)}$ is a smooth K3 surface defined over $k(\tau)$, which is a double cover of $\mathbb{P}^2$ ramified over the singular sextic curve $C+Q_\tau$, where we consider $C$ and $Q_\tau$ as plane curves defined over $k(\tau)$.

If one of the intersection points $C\cap Q_\tau$ were defined over $k(\tau)$, the strict transforms of the lines in $\mathbb{P}^2$ passing through this point would give an elliptic fibration on $Y_{k(\tau)}$. But this is not the case (cf. Example~\ref{exam:char5}). However, we can consider a degree 5 field extension $k(\tau)\subset k(C)$, so that the resulting K3 surface $Y_{k(C)}$ has this property --- it is elliptic over $k(C)$. In the geometric setting, this gives a dominant degree-$5$ map from an elliptic threefold (with rational base) to our sextic double solid $X$.

This construction can be extended to the case $d=\mathrm{deg}(C)>1$ as follows. A general line $\ell\subset\mathbb P^3$ meets $S$ at six distinct points, and its preimage in $X$ is a genus-two curve. If $\ell$ is tangent to $S$, then its preimage is a singular curve, and its normalization is an genus-one curve, in general. Using $C$, we can obtain a two-dimensional family of such lines: first we choose a point $p\in C$, and then we choose a line $\ell$ through $p$ in the embedded tangent plane $T_p(S)$. Then the normalized preimages of these lines in $X$ form a family of genus-one curves over a surface $B_C$, which is birational to $C\times\mathbb{P}^1$. Namely, we have 
$$
B_C=\{(p,\ell)\ :\ p\in C,\ p\in\ell\subset T_p(S)\},
$$
so $B_C=\mathbb P(T_S|_C)$, where projectivization parametrizes one-dimensional tangent directions.

Letting $E_{p,\ell}$ be the normalization of the preimage of the line $\ell$ in $X$, we obtain a family
$W_C\to B_C$ whose general fibers are genus one curves~$E_{p,\ell}$. The threefold $W_C$ is equipped with a natural rational map  $W_C\dashrightarrow X$, which is dominant and has degree $5d$, where the factor $5$ comes from the Gauss map of the surface $S$. 

Since there are two points in $E_{p,\ell}$ above the tangency point $p$, we obtain a double cover $M_C\to B_C$ such that the lifted family of genus one curves admits two sections. If $C$ is a line, then $M_C$ is birational to $C\times\mathbb P^1$. In general, the generic fiber of the induced map $M_C\dasharrow C$ is a conic over $k(C)$, so, by Tsen's theorem, $M_C$ is birational to the product $C\times\mathbb P^1$ over a finite extension of the base field $k$. In particular, if the curve $C$ is rational, then the surface $M_C$ is geometrically rational. 

Pulling back the family $W_C\to B_C$ via the double cover $M_C\to B_C$, we obtain a new family $Y_C\to M_C$ of genus one curves, which admits two sections. Summarizing, we obtain the following commutative diagram:
$$
\begin{tikzcd}[column sep=large,row sep=large]
Y_C \arrow[r,"2"] \arrow[d] &
W_C \arrow[r,dashed,"5d"] \arrow[d,"\psi"] & X\\
M_C \arrow[r,"2"'] & B_C &
\end{tikzcd}
$$
where the labels are just generic degrees. Consequently, the degree of the constructed dominant map 
$Y_C\dashrightarrow X$ is $10d$.

By construction, the family of genus one curves $Y_C\to M_C$ has two sections, and the rational points on $M_C$ are potentially dense. The triple tangent condition ensures that the difference of these sections gives a nontorsion divisor on the generic fiber of the family. Finally, using \cite{BT-elliptic} (see also \cite[Theorem~3.2]{Suzuki-2022}), 
we conclude that rational points on $Y_C$ are potentially dense, so they are potentially dense on $X$ as well.
\end{proof}

\begin{rema}
The moduli space of sextic double solids is $68$-dimensional. Counting parameters, one can compute the codimension in moduli imposed by the condition that the surface $S$ contains a curve $C$ of small degree. The following table summarizes the outcome in low-degree cases. 

\ 

\begin{center}\small
\begin{tabular}{|@{}|c|c|c@{}|}
\hline
Curve \(C\)&\(\mathrm{deg}(C)\)&Codimension in moduli\ $\,$\\
\hline
\hline
Line&1&3\\
\hline
Smooth conic&2&5\\
\hline
Twisted cubic&3&7\\
\hline
\ Smooth plane cubic&3&6\\
\hline
\ Smooth elliptic quartic&4&8\\
\hline
\end{tabular}
\end{center}
\end{rema}

\section{Ninety-five}
\label{sect:95}

The natural generalization of a smooth sextic double solid is a well-formed quasismooth Fano threefold hypersurface in weighted projective space
\begin{equation} 
\label{eqn:xd}
X_d\subset\mathbb P(a_1,a_2,a_3,a_4,a_5)
\end{equation}
such that $X_d$ has at most terminal singularities (so $X_d$ is a Mori fiber space over a point) and has Fano index $1$, i.e., 
$$
d=\sum_{i=1}^{5}a_i-1.
$$
These have been classified by Fletcher and Reid into $95$ deformation families. In that classification scheme,  No. 1 is the family of quartic threefolds 
$$
X_4\subset \bP^4=\bP(1,1,1,1,1),
$$
Family No. 2 consists of hypersurfaces of degree $5$ in $\mathbb{P}(1,1,1,1,2)$, and the sextic double solids form Family No. 3. For explicit description of the remaining families, we refer the reader to, e.g., \cite{AP-Fano} and \cite{ACP-2021}. 

Every quasismooth member of these 95 families is birationally rigid \cite{CPR,CP-95}. Potential density of rational points on these hypersurfaces has been studied in \cite{CP}, again using birationality to elliptically fibered threefolds. We extend these results as follows: 

\begin{theo}
\label{theo:93}
For $93$ out of $95$ deformation families of well-formed quasismooth weighted Fano threefold hypersurfaces there exist 
well-formed quasismooth members over $\mathbb Q$ with potentially dense sets of rational points. The exceptional families are:
\begin{enumerate}
\item Family No.~$75$: $X_{30}\subset\mathbb P(1,4,5,6,15)$;
\item Family No.~$93$: $X_{50}\subset\mathbb P(1,7,8,10,25)$.
\end{enumerate}
\end{theo}

We did not find a construction for the exceptional families in Theorem~\ref{theo:93}. 
We now sketch the proof of this theorem:
For $30$ deformation families (out of $95$), the assertion of this theorem follows from \cite[Proposition~1.4]{CP}. Many other families contain a well-formed quasismooth hypersurface $X_d$ as in \eqref{eqn:xd}, that admits a birational transformation into an elliptic fibration. Moreover, $X_d$ can be chosen to contain a rational surface with potentially dense set of rational points, which gives a multisection of the elliptic fibration, nontorsion under suitable genericity conditions. The elliptic fibration argument in the proof of Theorem~\ref{thm:mainn} implies potential density of rational points.

\begin{rema}
\label{rema:60}
General members of Families No. 3, 60, 75, 84, 87, 93 are not birational to an elliptic fibration by \cite[Theorem~1.2]{CP}. However, Family No. 60 contains special quasismooth members that are birational to elliptic fibrations. These form a codimension-$1$ subfamily. Concretely,  a 
quasismooth hypersurface in Family No. 60, 
$X_{24}\subset\mathbb P(1,4,5,6,9)$, can be given by 
$$
x_4x_5^2+x_5f_{15}(x_1,x_2,x_3,x_4)+f_{24}(x_1,x_2,x_3,x_4)=0,
$$
where $f_{15}$ and $f_{24}$ are polynomials of degree $15$ and $24$, respectively.
If $f_{15}$ does not contain $x_3^3$, or equivalently $\{x_1=x_2=x_4=0\}\subset X_{24}$,
then the projection 
\[
\begin{array}{ccc}    
X_{24} & \dashrightarrow & \mathbb P(1,4,6) \\
\left[x_1:x_2:x_3:x_4:x_5\right]& \mapsto & [x_1:x_2:x_4]
\end{array}
\]
gives a birational transformation of $X_{24}$ into an elliptic fibration.
Arguing as in \cite{CP,Cheltsov2007,CP-95}, one can show that this is essentially the only way to birationally transform $X_{24}$ into an elliptic fibration.
\end{rema}

Using birational transformations into elliptic fibrations, one can prove Theorem~\ref{theo:93} for $90$ families out of $95$. 
The remaining $5$ families are: 
\begin{equation} 
\label{eqn:fam}
\mathrm{Families \, No. \, 3,75,84,87,93}. 
\end{equation}
No quasismooth member of these families can be birationally transformed into an elliptic fibration, see Theorem~\ref{theo:elliptic} below. As mentioned above, Family No. $3$ consists of sextic double solids, which was dealt with in Theorem~\ref{thm:main}. A similar approach works for Families No.~$84$ and $87$. 
We consider these two cases in detail.

\begin{theo}[Family No. 84]
\label{thm:84}
There exists a codimension-2 subfamily of quasismooth hypersurfaces of degree $36$
$$
X_{36}\subset\mathbb{P}(1,7,8,9,12)
$$
consisting of threefolds with potentially dense sets of rational points, over any field of characteristic zero.  
\end{theo}

\begin{proof}
The hypersurface $X_{36}$ is given by 
\begin{multline*}
c_0x_5^3+x_5^2f_{12}(x_1,x_3)+c_1x_1^5x_2x_5^2+c_2x_1^3x_4x_5^2\\
+x_5f_{24}(x_1,x_3)+x_2x_5f_{17}(x_1,x_3)+x_2^2x_5f_{10}(x_1,x_3)+c_3x_1^3x_2^3x_5\\
+x_5x_4f_{15}(x_1,x_3)+x_2x_4x_5f_8(x_1,x_3)+c_4x_1x_2^2x_4x_5+c_5x_1^6x_5x_4^2\\
+c_6x_4^4+x_4^3f_9(x_1,x_3)+c_7x_1^2x_2x_4^3\\
+x_4^2f_{18}(x_1,x_3)+x_2x_4^2f_{11}(x_1,x_3)+c_8x_1^4x_2^2x_4^2\\
+x_4f_{27}(x_1,x_3)+x_2x_4f_{20}(x_1,x_3)+x_2^2x_4f_{13}(x_1,x_3)+c_9x_1^6x_2^3x_4\\
+f_{36}(x_1,x_3)+x_2f_{29}(x_1,x_3)+x_2^2f_{22}(x_1,x_3)\\
+c_{10}x_1^{15}x_2^3+c_{11}x_1^7x_2^3x_3+c_{12}x_1^8x_2^4+c_{13}x_2^4x_3+c_{14}x_1x_2^5=0,
\end{multline*}
where $c_0,\ldots,c_{14}$ are constants, and each $f_d(x_1,x_3)$ is a polynomial of degree $d$.
Assuming that $c_{13}\ne 0$, we may change coordinate $x_3$ to make $c_{14}=0$.
Then, the affine chart of $X_{36}$ given by $x_1=1$, is given by a quartic polynomial in $x_2,x_4,x_5$ whose coefficients are polynomials in $k[x_3]$. Hence, the generic fiber of the projection $X_{36}\dasharrow\mathbb{P}(1_{x_1},8_{x_3})$ is a quartic surface in $\mathbb{A}^3_{x_2,x_4,x_5}$ defined over $k(x_3)$. Compactifying this surface to the quartic surface $S_4\subset\mathbb{P}^3_{u,x_2,x_4,x_5}$, we see that the curve at infinity 
$$
C_{\infty}=S_4\cap \{u=0\}
$$ is given by
$$
u=x_2^2(c_3x_2+c_4x_4)x_5+c_6x_4^4+c_7x_2x_4^3+c_8x_2^2x_4^2+c_9x_2^3x_4+(c_{12}+c_{13}x_3)x_2^4=0.
$$
If we assume that $c_3=0$ and $c_4=0$ (codimension-2 condition) then 
$C_{\infty}$ is geometrically reducible and consists of four lines, over a finite extension of $k(x_3)$. Moreover, there exists a finite extension $K/k(x_3)$ such that $K\simeq k(x_3)$, and one of these four lines $L$ is defined over $K$. Projecting from $L$, we obtain an elliptic fibration $S_4\to\mathbb{P}^1$ defined over $K$, such that $L$ is a triple section of this fibration. 

Geometrically, this gives an elliptic fibration $\psi\colon Y\to Z$ over a rational surface, and a dominant rational map $Y\dasharrow X_{36}$ such that $\psi$ has a rational multisection of degree $3$. Imposing a generality conditions on $X_{36}$ guarantees that the multisection is nontorsion, we obtain the assertion as in \cite{BT-elliptic}.
\end{proof}

\begin{exam}
\label{exam:84}
The hypersurface $X_{36}\subset\mathbb P(1,7,8,9,12)$ given by
$$
x_1^{29}x_2+x_2^4x_3+x_3^3x_5+x_5^3+x_4^4=0.
$$
is a quasismooth member of Family No.~$84$. Arguing as in the proof of Theorem~\ref{thm:84}, we see that $X_{36}(\mathbb{Q})$ is Zariski dense. Indeed, in the chart $x_1=1$, make the substitutions 
$$
x_3=-s^4, \quad x_4=sx_2-\frac{u^3}{4},
$$
to obtain an affine $Y\subset\mathbb{A}^4$ and  a dominant rational map  $Y\dasharrow X_{36}$. 
Let $Y_{\mathbb{Q}(s,u)}\subset\mathbb{P}^2$ be the natural completion of the generic fiber of the projection to $\mathbb{A}^2_{s,u}$; it is  a smooth cubic curve, containing the points: 
\begin{align*}
O&=[su:1:0],\\
P&=[64u^4:u^{12}+64s^{12}u^4:512s^{13}u+8su^9-512].
\end{align*}
Their difference $P-O$ is nontorsion. It follows that  $Y(\mathbb{Q})$ is Zariski dense, by \cite{BT-elliptic}, thus  $X_{36}(\mathbb{Q})$ is dense as well.
\end{exam}

\begin{theo}[Family No. 87]
\label{thm:87}
Let $X_{40}\subset \mathbb P(1,5,7,8,20)$ be
a general quasismooth hypersurface of degree $40$. Then 
$X$ has a potentially dense set of rational points, over any field of characteristic zero. 
\end{theo}

\begin{proof}
Taking the affine chart $x_1=1$, projecting $X_{40}$ to $x_2$, and taking the generic fiber of this projection, we obtain a double cover of $\mathbb{A}^2_{x_3,x_4}$ ramified in a quintic curve defined over $k(x_2)$. The natural compactification of this surface is a sextic surface $S_6\subset\mathbb{P}(1_{x_3},1_{x_4},1_u,3_{x_5})$, containing the curve $\{u=x_5=0\}$. This curve contains five singular points of $S_6$, but none of them is defined over $k(x_2)$. Consider a degree $5$ extension of fields $K/k(x_2)$ such that $K\simeq k(x_2)$, and one of these singular points is defined over $K$. As in \cite{BT-K3}, blowing up $S_6$ at this point, we obtain an elliptic fibration over $\mathbb{P}^1$ with exceptional $(-2)$-curve as a double section. Arguing as in the proof of Theorem~\ref{thm:84}, we get the required assertion.
\end{proof}



\begin{exam}
\label{exam:87}
Let $X_{40}\subset\mathbb P(1,5,7,8,20)$ be the hypersurface given by
$$
x_5^{2}-x_4^{5}+x_2x_3^{5}+x_2^{8}+x_1^{40}+x_1^{33}x_3=0,
$$
it is quasismooth. 
In the affine chart $x_1=1$, substitute 
$$
x_2=s^5, \quad x_4=sx_3+5u^2.
$$
The resulting threefold $Y\subset\mathbb{A}^4$ is given by 
$$
x_5^2=(sx_3+5u^2)^5-s^5x_3^5-s^{40}-1-x_3,
$$
the corresponding rational map $Y\dasharrow X_{40}$ is dominant and has generic degree $10$.
Let $Y_{\mathbb{Q}(s,u)}\subset\mathbb{P}(1,1,2)$ be the natural completion of the generic fiber of the projection to $\mathbb{A}^2_{s,u}$. Then $Y_{\mathbb{Q}(s,u)}$ is a smooth elliptic curve, with points 
\begin{align*}
O&=[1:0:5s^2u],\\
P&=[1:0:-5s^2u]
\end{align*}
whose difference $P-O$ is nontorsion. By \cite{BT-elliptic}, 
$Y(\mathbb{Q})$ is Zariski dense, and thus $X_{40}(\mathbb{Q})$ as well.
\end{exam}

As mentioned in Remark~\ref{rema:60},  general members of Families in \eqref{eqn:fam} are not birational to an elliptic fibration. We conclude the paper by generalizing this result to {\em every} quasismooth member in these families.  A similar assertion for Family No. 3 (sextic double solids) is proved in \cite{CP2010,CP-2007}; but it fails for Family No. 60, by Remark~\ref{rema:60}.

\begin{theo}[{cf. \cite{Ryder06,Ryder,CP,Cheltsov2007,CheltsovPark-Halphen,CK}}]
\label{theo:elliptic}
Let 
$$
X_d\subset\mathbb P(a_1,a_2,a_3,a_4,a_5)
$$ 
be a quasismooth member of one of Families No. 75, 84, 87, 93.
Then $X_d$ is not birational to an elliptic or an abelian surface 
fibration.
\end{theo}

\begin{proof}
We may assume that $k$ is algebraically closed. Recall that:
$$
\begin{array}{c|c}
\text{Family No. 75}&X_{30}\subset\mathbb P(1,4,5,6,15)\\
\text{Family No. 84}&X_{36}\subset\mathbb P(1,7,8,9,12)\\
\text{Family No. 87}&X_{40}\subset\mathbb P(1,5,7,8,20)\\
\text{Family No. 93}&X_{50}\subset\mathbb P(1,7,8,10,25).
\end{array}
$$
Put $X=X_d$ and $A=-K_X$. 

Suppose that there exists a birational map $\phi\colon V\dasharrow X$ such that there exists a fibration $\psi\colon V\to B$ whose general fiber is an elliptic curve or an abelian surface. We let $\mathcal{M}_V=|\psi^*(H)|$ for a very ample divisor on $B$, we let $\mathcal{M}_X$ be the strict transform of the linear system $\mathcal{M}_V$ on $X$ via the map $\phi^{-1}$. If $\psi$ is an elliptic fibration, then $\mathcal{M}_X$ is not composed from a pencil. Vice versa, if $\psi$ is an abelian fibration, then $B=\mathbb{P}^1$ and we may take $H=\mathcal{O}_{B}(1)$, so that $\mathcal{M}_X$ is actually a pencil. In both cases, let $n\in\mathbb{N}$ such that
$$
\mathcal{M}_X\sim nA.
$$ 
Then $(X,\frac{1}{n}\mathcal{M}_X)$ is not terminal by \cite[Proposition~2.8]{CP}, and this log pair is canonical by \cite[Theorem~1.1.10]{CP-95}. We seek a contradiction.

The natural projection $\pi_{x_5}\colon X\to\mathbb P(1,a_2,a_3,a_4)$ is a finite morphism. 
Thus, if an irreducible curve $C$ were a center of non-terminal singularities of  $(X,\frac{1}{n}\mathcal{M}_X)$, then \cite[Lemma~2.1]{Ryder} would give
$$
A\cdot C\leqslant A^3,
$$ 
contradicting \cite[Lemma~2.2(a)]{Ryder}, because $d<a_2a_5$ and $\pi_{x_5}$ is finite. Thus,  $(X,\frac{1}{n}\mathcal{M}_X)$ has terminal singularities outside of finitely many points.

Let $p\in X$ be a point such that $(X,\frac{1}{n}\mathcal{M}_X)$ is not terminal at $p$. Put 
$$
b=\operatorname{lcm}(a_2,a_3,a_4).
$$
Then the  linear system $|\pi_{x_5}^*(\mathcal O_{\mathbb P(1,a_2,a_3,a_4)}(b))|$ is base point free, and its linear subsystem consisting of surfaces passing through $p$ does not have base curves. Computing, we obtain:
$$
\begin{array}{c|rrr}
\text{Family No.}&A^3&b&bA^3\\ \hline
75&1/60&60&1\\
84&1/168&504&3\\
87&1/140&280&2\\
93&1/280&280&1
\end{array}
$$

If $p\in X$ were a smooth point, then Pukhlikov's $4n^2$-inequality would give a contradiction, as in \cite[\S2.1]{CP-95}. Hence, $p$ is singular, in fact,  
a quotient singularity of type $\frac{1}{r}(1,a,r-a)$, where the possible values of $a$ and $r$ can be found in \cite{CPR,CP,CP-95} and in the table below. As in these papers, let $f:Y\to X$ be its Kawamata blowup, with exceptional divisor $E$, let $B=-K_Y$, and let $\mathcal{M}_Y$ be the strict transform of $\mathcal{M}_X$ on  $Y$.
Then $\mathcal M_Y\sim_{\mathbb Q}nB$, by Kawamata's theorem \cite{Kawamata}.

Let $S$ be the strict transform of $\{x_1=0\}\cap X$. Then $S$ is irreducible,
and $S\sim_{\mathbb{Q}} B$ or $S\sim_{\mathbb{Q}}B-E$, see the table below. Since $\mathcal M_Y$ is mobile, $B\cdot S$ is represented by an effective cycle. We will also choose an additional surface $D_X\subset X$, and put $D$ for its strict transform on $Y$. Then 
$$
D\sim_{\mathbb{Q}} cB+eE \sim_{\mathbb{Q}} cf^*(A)-\frac mrE, 
$$
for some $c>0$, $e\geqslant 0$, and $m>0$. The table below describes $D_X$,
as well as $c$ and $e$. In the last column, $\alpha$ is the coefficient of the monomial $x_2^4x_3$ ($x_2^6x_3$, respectively) in the defining equation of $X$ for Family~No. 84 (Family~No. 93, respectively). In the column labeled $\bigstar$, Yes stands for a technical condition, which will be explained later. 

{\tiny
\setlength{\tabcolsep}{7.5pt}
\begin{longtable}{@{}c l c c c l c l@{}}
\caption{Families 75, 84, 87 and 93.}\label{tab:families}\\
\toprule
Family No. & Point $p$ & $B^3$ & $S$ & $D$ & $D_X$ & $\bigstar$ & Condition\\
\midrule
\endfirsthead
\multicolumn{8}{@{}l}{\emph{Table~\ref{tab:families} (continued)}}\\
\toprule
Family No. & Point $p$ & $B^3$ & $S$ & $T$ & $T_X$ & $\bigstar$ & Condition\\
\midrule
\endhead
\midrule
\multicolumn{8}{r@{}}{\emph{continued on next page}}\\
\endfoot
\bottomrule
\endlastfoot
\multirow{4}{*}{75}
  & $\qs{4}{1,1,3}$             & $-\tfrac{1}{15}$    & $B$   & $5B+E$  & $\langle x_1x_2,\,x_3\rangle$  & No & \\*
  & $\qs{3}{1,1,2}$          & $-\tfrac{3}{20}$    & $B$   & $5B+E$  & $x_3=0$                    & Yes & \\*
  & $\qs{5}{1,4,1}$  & $-\tfrac{1}{30}$    & $B$   & $4B$    & $x_2=0$                    & Yes & \\*
  & $\qs{2}{1,1,1}$  & $-\tfrac{29}{60}$   & $B$   & $5B+2E$ & $x_3=0$                    & Yes & \\
\midrule
\multirow{5}{*}{84}
  & $\qs{8}{1,7,1}$             & $-\tfrac{1}{84}$    & $B$   & $7B$    & $x_2=0$                    & Yes & \\*
  & $\qs{7}{1,2,5}$             & $-\tfrac{1}{120}$   & $B$   & $8B$    & $x_3=0$                    & Yes & $\alpha\neq0$\\*
  & $\qs{7}{1,2,5}$             & $-\tfrac{1}{120}$   & $B-E$ & $12B+E$ & $x_5=0$                    & Yes & $\alpha=0$\\*
  & $\qs{3}{1,1,2}$          & $-\tfrac{9}{56}$    & $B$   & $8B+2E$ & $x_3=0$                    & Yes & \\*
  & $\qs{4}{1,3,1}$          & $-\tfrac{13}{168}$  & $B$   & $7B+E$  & $x_2=0$                    & Yes & \\
\midrule
\multirow{3}{*}{87}
  & $\qs{7}{1,1,6}$             & $-\tfrac{1}{60}$    & $B$   & $5B$    & $x_2=0$                    & Yes & \\*
  & $\qs{4}{1,1,3}$          & $-\tfrac{8}{105}$   & $B$   & $7B+E$  & $x_3=0$                    & Yes & \\*
  & $\qs{5}{1,2,3}$  & $-\tfrac{11}{420}$  & $B$   & $7B+E$  & $\langle x_1^2x_2,\,x_3\rangle$& No & \\
\midrule
\multirow{5}{*}{93}
  & $\qs{8}{1,7,1}$             & $-\tfrac{1}{70}$    & $B$   & $7B$    & $x_2=0$                    & Yes & \\*
  & $\qs{7}{1,3,4}$             & $-\tfrac{1}{120}$   & $B$   & $8B$    & $x_3=0$                    & Yes & $\alpha\neq0$\\*
  & $\qs{7}{1,3,4}$             & $-\tfrac{1}{120}$   & $B-E$ & $10B+E$ & $x_4=0$                    & Yes & $\alpha=0$\\*
  & $\qs{5}{1,2,3}$          & $-\tfrac{5}{168}$   & $B$   & $8B+E$  & $x_3=0$                    & Yes & \\*
  & $\qs{2}{1,1,1}$          & $-\tfrac{139}{280}$ & $B$   & $7B+3E$ & $x_2=0$                    & Yes & \\
\end{longtable}
}

Whenever we have 
Yes in the $\bigstar$ column, it follows from \cite[\S~5.2]{CP-95} that
the effective one-cycle $\Gamma=S\cdot D$ has irreducible support, no component in $E$,
and satisfies $D\cdot\Gamma\leqslant 0$. In all these cases, the proof of \cite[Lemma~3.2.2]{CP-95} implies that $\Gamma$ generates the extremal ray of $\overline{\operatorname{NE}}(Y)$ that is different from the
ray contracted by $f$. Thus, if $\bigstar$ holds and $e>0$, the identity
$$
 [\Gamma]=c[B\cdot S]+e[E\cdot S]
$$
contradicts extremality, because the cycles $B\cdot S$ and $E\cdot S$ are effective.
If $e=0$, then $S\sim_{\mathbb{Q}} B$ and $D\sim_{\mathbb{Q}} cB$, so that 
$$
\Gamma\equiv cB^2,
$$
which gives $B\cdot\Gamma=cB^3<0$. If $\mathcal M_Y$ is not composed with a pencil, we obtain a contradiction using \cite[Lemma~2.1]{CP}, which excludes the elliptic fibration case. If $\mathcal M_Y$ is a pencil, then \cite[Theorem~2.2]{CheltsovPark-Halphen} implies that $\mathcal{M}_Y=|cB|$, i.e., $\mathcal{M}_Y$ is the pencil generated by $cS$ and $T$. However, the normalization of a general surface in $|cB|$ is birational to a K3 surface, which excludes the abelian surface fibration case.

In the remaining two cases (which do not satisfy condition $\bigstar$), $|D|$ is a pencil, and its base curves are strict transforms of 
\begin{align*}
C&=\{x_1=x_3=0\}\cap X,\\ 
C'&=\{x_2=x_3=0\}\cap X. 
\end{align*}
Note that $p\not\in C'$. Moreover, the curve $C$ is irreducible, and its strict transform on $Y$ is numerically equivalent to $B\cdot D$. Then $D$ is nef by \cite[Lemma~3.2.6]{CP-95}. We obtain a contradiction using \cite[Corollary~2.2]{CP}.
\end{proof}

	\bibliography{sextic-double}

@incollection{CPR,
 author = {Corti, A. and Pukhlikov, A. and Reid, M.},
 title = {Fano 3-fold hypersurfaces},
 booktitle = {Explicit birational geometry of 3-folds},
 isbn = {0-521-63641-8},
 pages = {175--258},
 year = {2000},
 publisher = {Cambridge: Cambridge University Press},
 language = {English},
 zbMATH = {1513901},
 Zbl = {0960.14020}
}

@article{Ryder06,
 author = {Ryder, D.},
 title = {Classification of elliptic and {{\(K3\)}} fibrations birational to some {{\(\mathbb Q\)}}-{Fano} 3-folds},
 fjournal = {Journal of Mathematical Sciences. University of Tokyo},
 journal = {J. Math. Sci., Tokyo},
 issn = {1340-5705},
 volume = {13},
 number = {1},
 pages = {13--42},
 year = {2006},
 language = {English},
 zbMATH = {5140672},
 Zbl = {1139.14013}
}

@article{CheltsovPark-Halphen,
 author = {Cheltsov, I. and Park, J.},
 title = {Halphen pencils on weighted {Fano} threefold hypersurfaces},
 fjournal = {Central European Journal of Mathematics},
 journal = {Cent. Eur. J. Math.},
 issn = {1895-1074},
 volume = {7},
 number = {1},
 pages = {1--45},
 year = {2009},
 language = {English},
 doi = {10.2478/s11533-008-0056-2},
 zbMATH = {5624213},
 Zbl = {1183.14022}
}

@article{Cheltsov2007,
 author = {Cheltsov, I.},
 title = {Elliptic structures on weighted three-dimensional {Fano} hypersurfaces},
 fjournal = {Izvestiya: Mathematics},
 journal = {Izv. Math.},
 issn = {1064-5632},
 volume = {71},
 number = {4},
 pages = {765--862},
 year = {2007},
 language = {English},
 doi = {10.1070/IM2007v071n04ABEH002375},
 zbMATH = {5231350},
 Zbl = {1135.14031}
}

@article{Ryder,
 author = {Ryder, D.},
 title = {The curve exclusion theorem for elliptic and {{\(K3\)}} fibrations birational to {Fano} 3-fold hypersurfaces},
 fjournal = {Proceedings of the Edinburgh Mathematical Society. Series II},
 journal = {Proc. Edinb. Math. Soc., II. Ser.},
 issn = {0013-0915},
 volume = {52},
 number = {1},
 pages = {189--210},
 year = {2009},
 language = {English},
 doi = {10.1017/S0013091506000770},
 zbMATH = {5522641},
 Zbl = {1162.14008}
}

@incollection{Fletcher,
 author = {Iano-Fletcher, A. R.},
 title = {Working with weighted complete intersections},
 booktitle = {Explicit birational geometry of 3-folds},
 isbn = {0-521-63641-8},
 pages = {101--173},
 year = {2000},
 publisher = {Cambridge: Cambridge University Press},
 language = {English},
 zbMATH = {1513900},
 Zbl = {0960.14027}
}

@misc{Reid,
 author = {Reid, M.},
 title = {Canonical 3-folds},
 year = {1980},
 language = {English},
 howpublished = {Journees de geometrie algebrique, {Angers}/{France} 1979, 273-310 (1980).},
 zbMATH = {3704739},
 Zbl = {0451.14014}
}

@article {BT-elliptic,
    AUTHOR = {Bogomolov, F. A. and Tschinkel, Yu.},
     TITLE = {On the density of rational points on elliptic fibrations},
   JOURNAL = {J. Reine Angew. Math.},
  FJOURNAL = {Journal f\"ur die Reine und Angewandte Mathematik. [Crelle's
              Journal]},
    VOLUME = {511},
      YEAR = {1999},
     PAGES = {87--93},
      ISSN = {0075-4102,1435-5345},
   MRCLASS = {14G05 (11G35 14J45)},
  MRNUMBER = {1695791},
MRREVIEWER = {Constantin\ D.\ Manoil},
       DOI = {10.1515/crll.1999.511.87},
       URL = {https://doi.org/10.1515/crll.1999.511.87},
}

@article{CK,
 author = {Cheltsov, I. and Karzhemanov, I.},
 title = {Halphen pencils on quartic threefolds},
 fjournal = {Advances in Mathematics},
 journal = {Adv. Math.},
 issn = {0001-8708},
 volume = {223},
 number = {2},
 pages = {594--618},
 year = {2010},
 language = {English},
 doi = {10.1016/j.aim.2009.08.020},
 url = {hdl.handle.net/20.500.11820/c0160fa6-e1f2-489c-a2d4-83995b12ec24},
 zbMATH = {5653479},
 Zbl = {1207.14013}
}

@incollection{Kawamata,
 author = {Kawamata, Yu.},
 title = {Divisorial contractions to 3-dimensional terminal quotient singularities},
 booktitle = {Higher dimensional complex varieties. Proceedings of the international conference, Trento, Italy, June 15--24, 1994},
 isbn = {3-11-014503-0},
 pages = {241--246},
 year = {1996},
 publisher = {Berlin: Walter de Gruyter},
 language = {English},
 zbMATH = {1054919},
 Zbl = {0894.14019}
}

@book{CP-95,
 author = {Cheltsov, I. and Park, J.},
 title = {Birationally rigid {Fano} threefold hypersurfaces},
 fseries = {Memoirs of the American Mathematical Society},
 series = {Mem. Am. Math. Soc.},
 issn = {0065-9266},
 volume = {1167},
 isbn = {978-1-4704-2316-2; 978-1-4704-3643-8},
 year = {2017},
 publisher = {Providence, RI: American Mathematical Society (AMS)},
 language = {English},
 doi = {10.1090/memo/1167},
 zbMATH = {6752116},
 Zbl = {1383.14004}
}

@article {B-Sho,
    AUTHOR = {Bogomolov, F. and Halle, L. H. and Pazuki, F.
              and Tanimoto, S.},
     TITLE = {Abelian {C}alabi-{Y}au threefolds: {N}\'eron models and
              rational points},
   JOURNAL = {Math. Res. Lett.},
  FJOURNAL = {Mathematical Research Letters},
    VOLUME = {25},
      YEAR = {2018},
    NUMBER = {2},
     PAGES = {367--392},
      ISSN = {1073-2780,1945-001X},
   MRCLASS = {14J32 (11G05 14D06 14G05)},
  MRNUMBER = {3826826},
MRREVIEWER = {Tyler\ L.\ Kelly},
       DOI = {10.4310/MRL.2018.v25.n2.a2},
       URL = {https://doi.org/10.4310/MRL.2018.v25.n2.a2},
}

@article {BT-K3,
    AUTHOR = {Bogomolov, F. A. and Tschinkel, Yu.},
     TITLE = {Density of rational points on elliptic {$K3$} surfaces},
   JOURNAL = {Asian J. Math.},
  FJOURNAL = {Asian Journal of Mathematics},
    VOLUME = {4},
      YEAR = {2000},
    NUMBER = {2},
     PAGES = {351--368},
      ISSN = {1093-6106,1945-0036},
   MRCLASS = {14G05 (11G05 11G35 14C22 14J10 14J28)},
  MRNUMBER = {1797587},
MRREVIEWER = {Matthew\ H.\ Baker},
       DOI = {10.4310/AJM.2000.v4.n2.a6},
       URL = {https://doi.org/10.4310/AJM.2000.v4.n2.a6},
}

@article {Campana,
    AUTHOR = {Campana, F.},
     TITLE = {Orbifolds, special varieties and classification theory},
   JOURNAL = {Ann. Inst. Fourier (Grenoble)},
  FJOURNAL = {Universit\'e{} de Grenoble. Annales de l'Institut Fourier},
    VOLUME = {54},
      YEAR = {2004},
    NUMBER = {3},
     PAGES = {499--630},
      ISSN = {0373-0956,1777-5310},
   MRCLASS = {14E05 (14D06 14J40 32Q57 35Q15)},
  MRNUMBER = {2097416},
MRREVIEWER = {Dan\ Abramovich},
       DOI = {10.5802/aif.2027},
       URL = {https://doi.org/10.5802/aif.2027},
}

@article {AP-Fano,
    AUTHOR = {Cheltsov, I. and Park, J.},
     TITLE = {Birationally rigid {F}ano threefold hypersurfaces},
   JOURNAL = {Mem. Amer. Math. Soc.},
  FJOURNAL = {Memoirs of the American Mathematical Society},
    VOLUME = {246},
      YEAR = {2017},
    NUMBER = {1167},
     PAGES = {v+117},
      ISSN = {0065-9266,1947-6221},
      ISBN = {978-1-4704-2316-2; 978-1-4704-3643-8},
   MRCLASS = {14E08 (14E07 14J30 14J45)},
  MRNUMBER = {3605615},
MRREVIEWER = {Constantin\ Shramov},
       DOI = {10.1090/memo/1167},
       URL = {https://doi.org/10.1090/memo/1167},
}

@incollection{CP2010,
 author = {Cheltsov, I. and Park, J.},
 title = {Sextic double solids},
 booktitle = {Cohomological and geometric approaches to rationality problems. New Perspectives},
 isbn = {978-0-8176-4933-3; 978-0-8176-4934-0},
 pages = {75--132},
 year = {2010},
 publisher = {Boston, MA: Birkh{\"a}user},
 language = {English},
 doi = {10.1007/978-0-8176-4934-0_4},
 zbMATH = {5800254},
 Zbl = {1200.14031}
}

@article{CP-2007,
 author = {Cheltsov, I. and Park, J.},
 title = {Two remarks on sextic double solids},
 fjournal = {Journal of Number Theory},
 journal = {J. Number Theory},
 issn = {0022-314X},
 volume = {122},
 number = {1},
 pages = {1--12},
 year = {2007},
 language = {English},
 doi = {10.1016/j.jnt.2006.03.001},
 zbMATH = {5124502},
 Zbl = {1111.14010}
}

@article {CP,
    AUTHOR = {Cheltsov, I. and Park, J.},
     TITLE = {Weighted {F}ano threefold hypersurfaces},
   JOURNAL = {J. Reine Angew. Math.},
  FJOURNAL = {Journal f\"ur die Reine und Angewandte Mathematik. [Crelle's
              Journal]},
    VOLUME = {600},
      YEAR = {2006},
     PAGES = {81--116},
      ISSN = {0075-4102,1435-5345},
   MRCLASS = {14E05 (14E07 14G05 14J45)},
  MRNUMBER = {2283799},
MRREVIEWER = {Daniel\ J.\ Ryder},
       DOI = {10.1515/CRELLE.2006.087},
       URL = {https://doi.org/10.1515/CRELLE.2006.087},
}

@article {DS,
    AUTHOR = {de Jong, A. J. and Starr, J.},
     TITLE = {Every rationally connected variety over the function field of
              a curve has a rational point},
   JOURNAL = {Amer. J. Math.},
  FJOURNAL = {American Journal of Mathematics},
    VOLUME = {125},
      YEAR = {2003},
    NUMBER = {3},
     PAGES = {567--580},
      ISSN = {0002-9327,1080-6377},
   MRCLASS = {14D06 (14G27)},
  MRNUMBER = {1981034},
MRREVIEWER = {M.\ Kh.\ Gizatullin},
       URL =
              {http://muse.jhu.edu/journals/american_journal_of_mathematics/v125/125.3jong.pdf},
}

@article {GHS,
    AUTHOR = {Graber, T. and Harris, J. and Starr, J.},
     TITLE = {Families of rationally connected varieties},
   JOURNAL = {J. Amer. Math. Soc.},
  FJOURNAL = {Journal of the American Mathematical Society},
    VOLUME = {16},
      YEAR = {2003},
    NUMBER = {1},
     PAGES = {57--67},
      ISSN = {0894-0347,1088-6834},
   MRCLASS = {14M20 (14D05 14D22)},
  MRNUMBER = {1937199},
MRREVIEWER = {Alexandr\ V.\ Pukhlikov},
       DOI = {10.1090/S0894-0347-02-00402-2},
       URL = {https://doi.org/10.1090/S0894-0347-02-00402-2},
}

@article {Harris-T,
    AUTHOR = {Harris, J. and Tschinkel, Yu.},
     TITLE = {Rational points on quartics},
   JOURNAL = {Duke Math. J.},
  FJOURNAL = {Duke Mathematical Journal},
    VOLUME = {104},
      YEAR = {2000},
    NUMBER = {3},
     PAGES = {477--500},
      ISSN = {0012-7094,1547-7398},
   MRCLASS = {14G05 (11G35 14G25 14J20)},
  MRNUMBER = {1781480},
       DOI = {10.1215/S0012-7094-00-10436-X},
       URL = {https://doi.org/10.1215/S0012-7094-00-10436-X},
}

@article {HT-symm,
    AUTHOR = {Hassett, B. and Tschinkel, Yu.},
     TITLE = {Abelian fibrations and rational points on symmetric products},
   JOURNAL = {Internat. J. Math.},
  FJOURNAL = {International Journal of Mathematics},
    VOLUME = {11},
      YEAR = {2000},
    NUMBER = {9},
     PAGES = {1163--1176},
      ISSN = {0129-167X,1793-6519},
   MRCLASS = {14D06 (11G35 14D05 14J28 14K15)},
  MRNUMBER = {1809306},
MRREVIEWER = {Tam\'as\ Szamuely},
       DOI = {10.1142/S0129167X00000544},
       URL = {https://doi.org/10.1142/S0129167X00000544},
}

@misc{masa,
      title={Unirational {F}ano threefolds}, 
      author={Massarenti, A.},
      year={2026},
      note={{\tt arXiv:2608.22333}},
}

@incollection{ACP-2021,
  author = {Abban, H. and Cheltsov, I. and Park, J.},
  title = {On geometry of {Fano} threefold hypersurfaces},
  booktitle = {Rationality of Varieties},
  series = {Progress in Mathematics},
  volume = {342},
  pages = {1--14},
  publisher = {Birkh{\"a}user},
  year = {2021},
  doi = {10.1007/978-3-030-75421-1_1},
  eprint = {2007.14213},
  archivePrefix = {arXiv}
}

@article{Suzuki-2022,
  author = {Suzuki, F.},
  title = {Higher-dimensional {Calabi--Yau} varieties with dense sets of rational points},
  journal = {European Journal of Mathematics},
  volume = {8},
  number = {1},
  pages = {193--204},
  year = {2022},
  doi = {10.1007/s40879-021-00506-4}
}
	\bibliographystyle{alpha}
	
\end{document}